\documentclass[11pt]{amsart}
\usepackage{amsmath}
\usepackage{amssymb}
\usepackage{amscd}
\usepackage{url}
\usepackage[latin2]{inputenc}
\usepackage[T1]{fontenc}
\usepackage{pstricks}
\usepackage{pst-node}
\usepackage{pst-coil}
\usepackage{graphicx}

\newcommand{\CC}{\mathbb{C}}

\newcommand{\intpart}[1]{\left\lfloor#1\right\rfloor}

\newcommand{\bp}{\begin{pmatrix}}
\newcommand{\ep}{\end{pmatrix}}

\DeclareMathOperator{\ord}{ord}

\DeclareMathOperator{\sign}{sign}

\numberwithin{equation}{section}
\numberwithin{equation}{subsection}

\theoremstyle{plain}

\newtheorem{lemma}[equation]{Lemma}
\newtheorem{proposition}[equation]{Proposition}
\newtheorem{corollary}[equation]{Corollary}

\newtheorem{convention}[equation]{Convention}

\theoremstyle{definition}
\newtheorem{example}[equation]{Example}
\newtheorem{remark}[equation]{Remark}

\newtheorem{definition}[equation]{Definition}

\numberwithin{equation}{section}
\numberwithin{equation}{subsection}

\newtheorem*{acknowledgements}{Acknowledgements}

\newcommand{\mv}{\mathcal{V}}
\newcommand{\mvt}{\widetilde{\mathcal{V}}}
\newcommand{\mw}{\mathcal{W}}

\newcommand{\sm}[1]{{#1}^{*^{\scriptstyle -1}}}

\newcommand{\nd}{\textrm{ndeg}}

\title[Link invariants]{Hodge--type structures as link invariants}
\author{Maciej Borodzik}
\address{Institute of Mathematics, University of Warsaw, ul. Banacha 2,
02-097 Warsaw, Poland}
\email{mcboro@mimuw.edu.pl}
\thanks{The first author is supported by Polish MNiSzW Grant No N N201 397937 and also by a Foundation for Polish Science FNP. The second author is partially supported by OTKA Grant K67928. }
\author{Andr\'as N\'emethi}
\address{A. R\'enyi Institute of Mathematics, 1053 Budapest,  Re\'altanoda u. 13-15,  Hungary.}
\email{nemethi@renyi.hu}

\date{\today}
\makeatletter
\@namedef{subjclassname@2010}{\textup{2010} Mathematics Subject Classification}
\makeatother

\begin{document}
\begin{abstract}
Based  on some analogies with the Hodge theory of isolated hypersurface singularities,
we define Hodge--type numerical invariants of any, not necessarily algebraic, link in $S^3$.
We call them \emph{H--numbers}.
They contain the same amount of information as the
(non degenerate part of the) real Seifert matrix.  We study their basic properties,  and
we express the Tristram--Levine signatures and the higher order
Alexander polynomial in terms of them.
Motivated by singularity theory, we also introduce the \emph{spectrum} of the link (determined from these
$H$--numbers),  and we establish some semicontinuity properties for it.
 These properties can be related with skein--type  relations, although
 they are not so precise as the classical skein relations.
\end{abstract}
\maketitle

\section{Introduction}
\subsection{}
Although a Seifert matrix of a link is not a link invariant itself, it allows to define many link
invariants, which are on the one hand very deep, and easy to compute on the other. These invariants include
the Alexander polynomial, the signature and the Tristram--Levine signatures. It might be quite surprising
that the signature and the Alexander polynomial, although both come from a Seifert matrix, have completely
different properties. For example, the signature detects mirrors, and estimates the four-genus, while the Alexander
polynomial estimates the three-genus and does not detect mirrors.

Apparently straying from the knot theory, let us consider a 
hypersurface singularity in $(\mathbb{C}^{n+1},0)$. We can then associate
many important objects with it, as the intersection form on
the middle homology of the Milnor fiber, the monodromy matrix or the variation operator.
These three objects, together
with the space they act on, constitute a so--called {\it variation structure}.
If the singularity is isolated, then its variation structure is determined
by the variation operator, which is equivalent with the non--degenerate
Seifert bilinear form
associated with the germ of the singularity. Up to \emph{real} equivalence,
each such  variation structure is built of some explicitly written indecomposable pieces. The number of times
each such piece occurs in a concrete variation structure, is encoded in the so--called {\it
mod 2 equivariant (primitive)
Hodge numbers}  associated with  the singular germ. The name is motivated by the fact that they
are, indeed,  mod 2 reductions of the equivariant Hodge numbers associated with the
mixed Hodge structure of the vanishing cohomology of the singularity, defined as in \cite{Stee}.
In this way one also sees that the information codified in these numbers is
equivalent with the real Seifert form \cite{Nem-real}.

From the mod 2 equivariant Hodge
 numbers many other invariants can be reread, like the characteristic polynomial of monodromy,
or different signature--type invariants.  In fact, if $n=1$,
then from them one  can even recover completely
all the equivariant Hodge numbers of the mixed Hodge structure of the vanishing cohomology.

The algebraic links (case $n=1$)
form a bridge between the singularity theory and the knot theory (but the correspondences
can be continued in higher dimensions, too).  The Alexander polynomial
of an algebraic link is exactly the characteristic polynomial of the monodromy of the corresponding Milnor fibration.
The Milnor fiber constitutes
 a natural Seifert surface of an algebraic link. The corresponding Seifert matrix is the transposed inverse
of the variation operator. In other words, the variation structure of (plane curve) singularities is deeply related to
link invariants of  algebraic links.

Motivated strongly by the case of algebraic links, we associate  a variation structure with  \emph{any}
link in $S^3$, and we define the analogs of
\emph{mod 2 equivariant (primitive)  Hodge numbers}. The variation structure
is built of the real Seifert matrix of the link and  determines the Seifert matrix up to real S-equivalence.
In fact, the newly defined  numbers codify and determine this structure.
Although the present work does not contain any Hodge theoretical discussion, motivated
by the above correspondence we still call the introduced numbers \emph{H--numbers}.
Actually, the nilpotent part of the suitably defined 
monodromy operator defines a weight filtration, and also one can define
a mod 2--Hodge filtration similarly as in \cite{Nem-real}, hence a `mod 2 Hodge structure' exists 
(it would be interesting to extend it to a genuine Hodge structure).

Both the (higher) Alexander polynomial and the Tristram--Levine signatures can  be easily  expressed in terms  of the
H--numbers of the link. Their symmetries  and their behaviour  under taking
mirrors,  allow us to explain e.g.
why  the Alexander polynomial does not distinguish mirrors, while the signatures do.

In the Hodge theory of hypersurface singularities, the (numerical part) of the mixed Hodge structure
 was codified by
Steenbrink and Varchenko in the so--called spectrum. This codification was motivated by the
extremely powerful and mysterious  semicontinuity behavior of it under the deformation.

In our present context we also introduce the {\it mod 2 spectrum} of a link in $S^3$, and we relate it with
the classical link invariants as the higher Alexander polynomials and Tristram--Levine signatures.
The relation between the spectrum and Tristram--Levine signatures (see Proposition~\ref{p:signatures}) is one
of the key ingredients of showing semicontinuity results for spectra by topological methods as in \cite{BN}.
It also emphasizes the unifying power of the newly introduced invariants, which gather together more
conceptually all the classical properties and invariants associated with real Seifert matrices.
On the other hand,  since
we know only the classification of the variation structures over reals, and  not over integers,
in this discussion we loose some information regarding the integer  Seifert matrices, like the determinant.

After the foundations, we try to alloy the two main strategies used in the two theories:
the technique of skein relations of classical link theory with the semicontinuity of the spectrum
(known in singularity theory). Although at the origin and substance of both sits surgery,
for the second case one needs a special surgery with intrinsic  monotonicity structure
 (this, in the singularity theory, is  guaranteed by the
presence of the deformation). In our results we will assume the monotonicity of the degree of the
Alexander polynomial.

The basic motivation for studying semicontinuity of our structures is the following. First, we believe that the newly
introduced H--numbers 
do not admit so precise skein relations similarly as some of the classical link invariants; or, their
form should be packed in a more intelligent way. We believe that this `packing' goes through
the spectrum, and  the corresponding semicontinuity relations will guide the corresponding surgery formulas.
Similarly as in the case of analytic singularities, where the semicontinuity had remarkable applications
(see e.g. \cite{Var}), we expect in the future similar consequences for the newly introduced spectrum too.
For more comments  see  (\ref{ss:semies}).

\subsection{}
The structure of the paper is the following. We begin with a definition and examples of variation structures in
Section~\ref{s:variation}. Then we recall the classification theorem of \cite{Nem-real} regarding real, simple (see Definition~\ref{def:simplehodge})
variation structures,  and define the  H--numbers
for links. In Section~\ref{s:linkinv} we relate the classical link invariants (higher Alexander polynomials,
rational Nakanishi index, Tristram--Levine signatures) to the H--numbers. In the next section we  show
some examples. In the last section we gather  results about the skein relation,
the proofs mostly go through skein relation for the Alexander polynomial or the signatures.
Also, we establish some semicontinuity results for the spectrum. By them we wish to draw the attention of the
readers to this new phenomenon with the hope that this will bring some deep and powerful
instrument in the near future.

\begin{acknowledgements}
The authors wish to express their thanks to R\'enyi Institute for hospitality, to L.~Kaufmann, A.~Stoimenow, P.~Traczyk,
H.~Trotter and H.~Zoladek
for many fruitful discussions on the subject,  and to S.~Friedl for  pointing out the relation
of the newly discussed  invariants with the Nakanishi index.
\end{acknowledgements}

\section{Variations structures. Definitions and examples}\label{s:variation}
\subsection{Definitions}
Here we recall some definitions from \cite[Section 2]{Nem-Var}. We begin with fixing
some standard notation.

For a finite dimensional complex vector space $U$, we denote its dual by $U^*$. The natural
isomorphism $\theta\colon U\to U^{**}$ is given by $\theta(u)(\phi)=\phi(u)$. The complex
conjugation is denoted by a bar $\bar{\cdot}$;  for any  $\phi:U\to V$ its dual map is denoted
by  $\phi^*:V^*\to U^*$. Let us also recall that,
if $\phi$ is represented by a matrix $S$ in some basis, then $\phi^*$ is represented by a transpose
$S^T$ in the dual basis. For a matrix $S$, $S_{kl}$ or $S_{k,l}$ denotes the coefficient of $S$ in $k$-th row and $l$-th column.

It is convenient to regard hermitian forms in the following way.
\begin{definition}
A $\mathbb{C}$--linear endomorphism $b\colon U\to U^*$ with $\overline{b^*\circ\theta}=\varepsilon b$
(where $\varepsilon=\pm 1$) is called \emph{$\varepsilon$--hermitian form} on $U$.
\end{definition}
Remark, that we do not assume here that $b$ is non--degenerate.
The automorphisms of $b$ consists of isomorphisms $h\colon U\to U$ preserving $b$, i.e. with $\bar{h}^*\circ b\circ h=b$.
\begin{definition}\label{d:HVS}
An \emph{$\varepsilon$--hermitian variation structure} (abbreviated by HVS) over $\mathbb{C}$ is
a quadruple $(U;b,h,V)$, where
\begin{itemize}
\item[(1)] $U$ is finite dimensional vector space over $\mathbb{C}$;
\item[(2)] $b\colon U\to U^*$ is an $\varepsilon$--hermitian form on $U$;
\item[(3)] $h\colon U\to U$ is a $b$--orthogonal automorphism of $U$;
\item[(4)] $V\colon U^*\to U$ is a $\mathbb{C}$--linear endomorphism such that
\begin{align*}
\overline{\theta^{-1}\circ V^*}&=-\varepsilon V\circ\overline{h}^*\\
V\circ b&=h-I.
\end{align*}
\end{itemize}
Here and afterwards $I$ denotes the identity map. The name of the structure is inherited from
the operator $V$, which usually is a `variation map', cf. (\ref{Milnor}).
$V$, respectively $h$ will be called {\it variation}, respectively {\it monodromy} operator.
\end{definition}
Observe that from (4) it follows immediately that
\begin{equation}\label{eq:immediateproperties}
b\circ V={\overline{h}^*}^{-1}-I\text{, and } h\circ V\circ\bar{h}^*=V.
\end{equation}
\begin{definition}\label{def:simplehodge}
The HVS $(U;b,h,V)$ will be called \emph{non--degenerate} (respectively \emph{simple})
if $b$ (respectively $V$) is an isomorphism.
\end{definition}
We will need following lemmas from \cite{Nem-real}:
\begin{lemma}
For a triple $(U;b,h)$ satisfying points (1)--(3)  from Definition~\ref{d:HVS}, if $b$ is non--degenerate,
then there exists a unique $V$, namely $V=(h-I)b^{-1}$,  such that $(U;b,h,V)$ constitutes a HVS.
\end{lemma}

The non--degenerate triplets $(U;b,h)$ are classified by Milnor \cite{Milnor-forms}, see also \cite{Neu0}.
\begin{lemma}\label{l:VtoHVS}
For a pair $(U;V)$ with $V$ an isomorphism, there exist unique $b$ and $h$ such that $(U;b,h,V)$ is a HVS.
Indeed, $h=-\varepsilon V\sm{\bar{V}}$ and $b=-V^{-1}-\varepsilon\sm{\bar{V}}$ satisfy the axioms.
\end{lemma}
From the last lemma it follows that the classification of simple HVS is equivalent to the classification
of $\mathbb{C}$--linear isomorphisms $V\colon U^*\to U$.

There is a natural notion of an isomorphism of a HVS:

\begin{definition}\label{def:dirsum}
\begin{itemize}
\item[(a)] Two HVS $(U;b,h,V)$ and $(U';b',h',V')$ are \emph{isomorphic}, denoted by $\simeq$,
 if there exists an isomorphism $\phi:U\to U'$
such that
$b=\bar{\phi}^*b'\phi$,
$h=\phi^{-1}h'\phi$,
and $V=\phi^{-1}V'  (\bar{\phi}^{*})^{-1}$.

\item[(b)]
If $(U_1;b_1,h_1,V_1)$ and $(U_2;b_2,h_2,V_2)$ are two HVS with the same sign $\varepsilon$, their \emph{direct sum} is given
by $(U_1\oplus U_2;b_1\oplus b_2,h_1\oplus h_2,V_1\oplus V_2)$. 
For sum of $m$ copies of $\mv$ we write $m\cdot \mv$.

\item[(c)] The \emph{conjugate} of $\mv=(U;b,h,V)$ is defined as
$\bar{\mv}=(U;\bar{b},\bar{h},\bar{V})$.
\end{itemize}
\end{definition}

\subsection{Examples and classification of HVS}
Here we shall follow closely \cite{Nem-real}, unless stated otherwise all results in this section are proved in \cite{Nem-real}.
For $k\ge 1$, $J_k$ denotes the $(k\times k)$--Jordan block
 with eigenvalue $1$.

\begin{example}\label{ex:mv2k}
For $\lambda\in\mathbb{C}^*\setminus S^1$ and $k\ge 1$, the quadruple
\[
\mv_{\lambda}^{2k}=\left(\mathbb{C}^{2k};%
\left(\begin{matrix}0&I\\\varepsilon I&0\end{matrix}\right),%
\left(\begin{matrix}\lambda J_k&0\\ 0&\frac{1}{\bar{\lambda}}{J_k^*}^{-1}\end{matrix}\right),%
\left(\begin{matrix}0&\varepsilon(\lambda J_k-I)\\\frac{1}{\bar{\lambda}}{J_k^*}^{-1}-I&0\end{matrix}\right)\right)
\]
defines a HVS. Moreover, $\mv_{\lambda}^{2k}$ and
$\mv_{1/\bar{\lambda}}^{2k}$ are isomorphic.
\end{example}

Before we show the next example we need a computational lemma;
here one needs to consider the two square roots of $\varepsilon$. The two canonical sign  choices for them are motivated by
Hodge theoretical
sign--conventions (cf. \cite{Nem-real}, Sections 5 and 6) of the variation structures associated with
isolated hypersurface singularities $(\CC^{n+1},0)\to (\CC,0)$, where $\varepsilon=(-1)^n$.

\begin{lemma}\label{l:BB}
For any $k>1$ there are precisely two non-degenerate $\epsilon$--hermitian forms
 (up to a real positive scaling), denoted by $b^k_\pm$, such that
\[\bar{b}^*=\varepsilon b\text{ and }J_k^*bJ_k=b.\]
By convention, the signs are fixed by $(b^k_{\pm})_{1,k}=\pm i^{-n^2-k+1}$, where $\epsilon=(-1)^n$. $b$ is
also \emph{left diagonal}, i.e., 
the entries of $b$ also satisfy: $b_{i,j}=0$ for $i+j\leq k$ and $b_{i,k+1-i}=(-1)^{i+1}b_{1,k}$.
\end{lemma}


\begin{example}\label{e:mvlambda}
Let $\lambda\in S^1$. Up to isomorphism there are two non--degenerate HVS such that $h=\lambda J_k$. These are
\[\mv^k_{\lambda}(\pm 1)=\left(\mathbb{C}^k;b^k_{\pm},\lambda J_k,(\lambda J_k-I)(b^k_\pm)^{-1}\right).\]
\end{example}

Notice  that these structures are simple unless $\lambda=1$.
In fact,
if $\lambda\neq 1$ then any HVS with $h=\lambda J_k$ is both non-degenerate and simple.
The case with eigenvalue $1$ admits also a pair of degenerate HVS.

\begin{lemma}\label{l:mvtilde}
For $k\ge 2$ there are two degenerate HVS with $h=J_k$. They are
\[\mvt^k_1(\pm 1)=\left(\mathbb{C}^k;\widetilde{b}_\pm,J_k,\widetilde{V}_{\pm}^k\right),\]
where
\[\widetilde{b}^k_\pm=\left(\begin{matrix} 0 & 0\\ 0&b^{k-1}_\pm \end{matrix}\right)\]
and $\widetilde{V}_\pm^k$ is uniquely determined by $b$ and $h$ (up to an isomorphism).
Moreover,   $\mvt^k_1(\pm 1)$  is simple.
In fact, the entries of $V^{-1}$ satisfy: $(V^{-1})_{i,j}=0$ for $i+j\geq k+2$, $(V^{-1})_{i,k+1-i}=
\pm (-1)^{i+1} i^{-n^2-k}$. In order to recognize the isomorphism type, we have to recognize these
entries up to a real positive re-scaling.
\end{lemma}

For $k=1$ (i.e. $U=\mathbb{C}$), and $h=I$, the structures can be written down more explicitly;
there are the following five $\varepsilon$--HVS's with  $\varepsilon=(-1)^n$:

\begin{align*}
\mv^1_1(\pm 1)&=(\mathbb{C},\pm i^{-n^2},I,0)\\
\mvt^1_1(\pm 1)&=(\mathbb{C},0,I, \pm i^{n^2+1})\\
\mathcal{T}&=(\mathbb{C},0,I,0).
\end{align*}

From all these examples the structures $\mv^k_1(\pm 1)$ and $\mathcal{T}$ are non--simple,
and $\mvt^1_1(\pm 1)$ are simple.
Concluding, for any $\lambda\in S^1$ and in each dimension $k$, there are precisely two non-equivalent simple
variation structures with $h=\lambda J_k$.
We use the following uniform notation for them:

\begin{equation}\label{eq:wk}
\mw^k_\lambda(\pm 1)=%
\begin{cases}
\mv^k_\lambda(\pm 1)&\text{if $\lambda\neq 1$}\\
\mvt^k_1(\pm 1)&\text{if $\lambda=1$.}
\end{cases}
\end{equation}

\begin{proposition}\label{p:classification}
A simple HVS is uniquely expressible as a sum of indecomposable ones up to ordering of summands and up to
an isomorphism. The indecomposable pieces are
\begin{eqnarray*}
 & \mw^k_\lambda(\pm 1)&\text{ for $k\ge 1$, $\lambda\in S^1$}\\ 
& \mv^{2k}_\lambda &\text{ for $k\ge 1$, $0<|\lambda|<1$.}\end{eqnarray*}

\end{proposition}

\begin{convention}
From now on, all HVS we shall discuss, are assumed to be simple.
\end{convention}

The above proposition allows us to define some invariants of HVS.

\begin{definition}\label{d:pklambda}
Let $\mv$ be a simple
HVS $\mv$. Let us express it, according to Proposition~\ref{p:classification},  
as
\begin{equation}\label{eq:mvsum}
\mv=
\bigoplus_{\substack{0<|\lambda|<1\\ k\ge 1}} q^k_\lambda\cdot   \mv^{2k}_\lambda
\oplus
\bigoplus_{\substack{|\lambda|=1\\ k\ge 1, \ u=\pm 1}} p^k_\lambda(u)\cdot  \mw^{k}_\lambda(u),
\end{equation}
where  the expression of type $r\cdot\mv$ is a shorthand for a sum $\mv\oplus\dots\oplus\mv$ ($r$ times), cf. (\ref{def:dirsum}).
Then $\{q^k_\lambda\}_{|\lambda|<1}$ and $\{p^k_\lambda(\pm 1)\}_{\lambda\in S^1}$
are called the \emph{H--numbers} of $\mv$.
\end{definition}

\begin{remark}
The above classification result is over $\mathbb{C}$ or, equivalently, over $\mathbb{R}$. One can
consider HVS's over $\mathbb{Z}$ as well, but then the classification is unknown.
\end{remark}

If $V$ is defined over the real numbers, then the above decomposition has some symmetries. Let $s$ be
$1$ if $\lambda=1$ and $0$ if $\lambda\in S^1\setminus \{1\}$. Then, with $\varepsilon=(-1)^n$,

\begin{equation}
\overline{\mv^{2k}_\lambda}=\mv^{2k}_{\bar{\lambda}} \ \ \ \mbox{for $\lambda\not\in S^1$, and}
\end{equation}
\begin{equation}
\overline{\mw^k_\lambda(\pm 1)}=\mw^k_{\bar{\lambda}}(\pm (-1)^{n+k+1+s}) \ \ \ \mbox{for $\lambda\in S^1$}.
\end{equation}

Therefore we have the following result

\begin{lemma}\label{l:symmetry}
If in the HVS $\mv$, the matrix $V$ is defined over reals, then
\[q^k_\lambda=q^k_{\bar{\lambda}} \ \ (\mbox{for} \ |\lambda|<1) \ \ \ \mbox{and} \ \ \
p^k_\lambda(\pm 1)=p^k_{\bar{\lambda}}(\pm (-1)^{n+k+1+s}) \ \ (\mbox{for} \ |\lambda|=1).\]
\end{lemma}

\noindent Moreover, by an easy check of the coefficient $b_{k,1}$ one has

\begin{lemma}\label{l:minus}
Let  $V$ be the variation operator of the simple structure $\mv$. Let $-\mv$ be the structure
corresponding to the variation operator  $-V$ (see Lemma~\ref{l:VtoHVS}). Then $-\mw^k_\lambda(\pm 1)\simeq\mw^k_\lambda(\mp 1)$, and  $-\mv^{2k}_\lambda\simeq\mv^{2k}_\lambda$.
\end{lemma}
\noindent One needs slightly more computations to verify:
\begin{lemma}\label{l:transpose}
Let  $V$ be the variation operator of the simple structure $\mv$. Let $\mv^T$ be the structure
determined by the variation operator  $V^T$.
Then $\mv^T$ and $\mv$ are isomorphic.
\end{lemma}

\begin{proof} The statement is clear for  $\mv=\mv^{2k}_\lambda$. Hence, assume that
$\mv=(U;b,h,V)$ is $\mw^k_\lambda(\pm 1)$. Consider  the new
structure with variation operator $W=V^T$.
Since for $\phi=V^T$ one has $\phi^{-1}V^T\overline{\phi}^{*,-1}=\overline{V}^{-1}$,
the variation structures associated with $W$ and $\overline{V}^{-1}$ are isomorphic.
The monodormy operator of $\overline{V}^{-1}$ is $-\varepsilon \overline{V}^{-1}V^*=
h^*=\lambda J_k^T$. Next, consider the anti-diagonal matrix $A$ with $A_{ij}=1$ if
$i+j=k+1$ and zero otherwise; it satisfies $A=A^{-1}=A^T=\overline{A}$. Base change by $A$ has the effect
$A(\lambda J_k^T)A=\lambda J_k$, hence the monodromy operator is twisted to a
`normal form', which agrees with $h$. In particular,
it is enough to compare the two
variation operators $V$ and $A\overline {V}^{-1}A$. If $\lambda=1$, by the last sentence of (\ref{l:mvtilde}),
it is enough to compare the anti-diagonals of these two operators, which clearly agree.
If $\lambda\not=1$ use (\ref{l:BB}) and the same type of  argument.
\end{proof}

\subsection{Spectrum and the extended spectrum}

One can extract from a variation structure a  weaker invariant, whose motivation will be
explained in the next subsection when we discuss the spectrum of an isolated hypersurface singularity.

\begin{definition}\label{d:mod2spec}
Let $\mv$ be a HVS. The \emph{$mod\,\,2$--spectrum} (or, shortly, the spectrum) of $\mv$ is
a finite set $Sp$ of real numbers from $(0,2]$ such that any real non-integer number $\alpha\in(0,2]$
occurs in $Sp$ precisely $s(\alpha)$ times, where
\[s(\alpha)=\sum_{\substack{k\text{ odd}\\u=\pm 1}}\frac{k-uv}{2}\cdot
p^{k}_\lambda(u)+\sum_{\substack{k\text{ even}\\u=\pm 1}}\frac{k}{2}\cdot p^k_\lambda(u),\]
where
\[e^{2\pi i\alpha}=\lambda\,\text{ and }(-1)^{\intpart{\alpha}}=v.\]
The H--numbers $p^k_1(\pm 1)$ correspond to elements $1$ and $2$ in the spectrum, appearing
precisely $k/2$ times each if $k$ is even, and $(k\pm 1)/2$ and $(k\mp 1)/2$ times if $k$ is odd.
\end{definition}

A consequence of Lemma~\ref{l:symmetry} is the following symmetry property
\begin{corollary}\label{l:symprop}
If $V$ is a real matrix, and $\epsilon =-1$,
then $Sp\setminus\mathbb{Z}$ is symmetric with respect to $1$.
\end{corollary}

Notice that $Sp$ contains no information regarding the blocks with eigenvalues $\lambda\not\in S^1$.
To enclose the information regarding $\{q^k_\lambda\}_{|\lambda|<1}$
we define the \emph{extended spectrum}.
Remark that, this construction has no counterpart in the classical Hodge theory.

\begin{definition}
The \emph{extended spectrum} $ESp$ of a HVS $\mv$ is a finite subset of complex numbers from $(0,2]\times i\mathbb{R}$
of the form $ESp=Sp\cup ISp$, where $ISp\cap\mathbb{R}=\emptyset$ and any non-real number $z\in (0,2]\times i\mathbb{R}$,
$z=x+iy$,
occurs in $ISp$ precisely $s(z)$ times, where
\[
s(z)=\begin{cases}
\sum k\cdot q^k_\lambda&\text{if $x\le 1$, $y>0$ and $e^{2\pi i z}=\lambda$}\\
\sum k\cdot q^k_\lambda&\text{if $x>1$, $y<0$ and $e^{2\pi i z}=1/\bar{\lambda}$}\\
0&\text{if $x\leq 1$ and $y<0$, or $x> 1$ and $y>0$}.
\end{cases}
\]
\end{definition}
In other words, a  block $\mv^{2k}_{\lambda}$ (where $|\lambda|<1$)
contributes $k$ times to both $x+iy$ and  $1+x-iy$, if
$e^{2\pi (-y+ix)}=\lambda$ and $x\in(0,1]$.

We have the following two important properties of $ESp$.
\begin{lemma}\label{l:ISp-nosig}
For any $u\in(0,1)$, let $H_u=(u,u+1)\times i\mathbb{R}$. Then, if $ISp\cap\partial H_u=\emptyset$,
we have
\[\#ISp\cap H_u=\#ISp\setminus H_u.\]
\end{lemma}
\begin{proof}
This follows directly from a simple observation that from the two numbers $x+iy$ and $1+x-iy$,
one of them lies in $H_u$, and one of them does not.
\end{proof}
\begin{lemma}\label{l:symprop2}
If $\mv$ is real  and $\varepsilon=-1$,
then $ESp\setminus\mathbb{Z}$ is symmetric (via point--reflection) with respect to $1$.
\end{lemma}

\begin{proof}
By Corollary~\ref{l:symprop}, it is enough to prove that $ISp$ is symmetric. But this follows from the fact that
$\lambda=e^{2\pi (-y+ix)}$ yields the  points $x+iy$ and $1+x-iy$ in $ISp$, while
$\bar{\lambda}=e^{2\pi (-y-ix)}$ the points $1-x+iy$ and $2-x-iy$.
\end{proof}

\subsection{Variations structures of Milnor fibers}\label{Milnor}
The motivation of the definition of HVS comes from the topological invariants of
complex isolated hypersurface singularities
and their relationship with the mixed Hodge structure on the vanishing cohomology.

Let $f:(\mathbb{C}^{n+1},0)\to(\mathbb{C},0)$ ($n\geq 0$)
be an analytic germ such that $f^{-1}(0)$ has an  isolated singularity at the origin.
Let $S^{2n+1}$ be a small sphere around $0$, $K=S^{2n-1}\cap\{f=0\}$ the link, and
\[\phi\colon S^{2n+1}\setminus K\to S^1, \ \ \ \ \phi(z)=f(z)/|f(z)|\]
the Milnor fibration  (see \cite{Mil-singular}) with
fiber $F=\phi^{-1}(1)$.
Set $F_t=\phi^{-1}(e^{2\pi i t})$ for $t\in [0,1)$ (with $F=F_0$).
Then the trivialization of the bundle over $[0,1)$ gives diffeomorphisms
(defined up to isotopy)  $\gamma_t\colon F\to F_t$
for $t\in [0,1)$, and extended to $t=1$, the geometric monodromy
$\gamma_1:F\to F$. They give rise to a well-defined map
\[\Gamma_t\colon \tilde{H}_n(F_1)\to \tilde{H}_n(F_t)\]
and the \emph{monodromy map}
\[h=\Gamma_1\colon \tilde{H}_n(F)\to \tilde{H}_n(F).\]

One also defines the intersection form on $b:\tilde{H}_n(F)^{\otimes 2}\to {\mathbb R}$
which is $(-1)^n$ symmetric.
Since  $\gamma_1$ is chosen such that it is identity on $\partial F$, one
also defines a \emph{variation} map $V\colon \tilde{H}_n(F,\partial F)\to \tilde{H}_n(F)$
(see \cite[Chapter 4.2]{Zo} or \cite[Chapter 1.2]{AVG}).
Here, by  Lefschetz duality one has the identification  $\tilde{H}_n(F,\partial F)\simeq Hom(\tilde{H}_n(F),{\mathbb R})$.
The  next fact  is well--known (see e.g. \cite{Nem-real}):
\begin{proposition}
The quadruple $(U=\tilde{H}_n(F,\mathbb{C})$, $b,h,V)$ 
form a HVS with $\varepsilon=(-1)^n$.
\end{proposition}

\begin{definition}\label{d:mvf}
The variation  structure defined above is called the \emph{variation structure of the singularity $f$} and
it is denoted by $\mv_{f}$.
\end{definition}

Notice that $\mv_f$ is defined over ${\mathbb R}$. Additionally,  it has some other particular properties as well. First of all, by the Monodromy Theorem (see e.g. \cite[Theorems 3.11 and 3.12]{AVG} or \cite[Chapter 7, \S 4]{Zo}), all the eigenvalues of $h$ are roots of unity.
Moreover,  the block--decomposition of $\mv_f$  is  closely related with the
mixed Hodge structure of $U$.

Recall (see e.g. \cite{Nem-real} for the facts below) that $U$ carries a mixed Hodge structure
compatible with the monodromy action.
Let us denote the corresponding equivariant Hodge numbers by $h^{a,b}_\lambda$.  The nilpotent part of the monodromy  defines a morphism of Hodge structures of type $(-1,-1)$, let us denote by $p^{a,b}_\lambda$ the dimensions of the corresponding primitive $\lambda$--generalized 
eigenspaces, which are, in general,  non--trivial for $ a+b\geq n+s$. 
Then
$$p^{a,b}_\lambda=h^{a,b}_\lambda-h^{a+1,b+1}_\lambda \ \ \mbox{and} \ \ \
h^{a,b}_\lambda=\sum_{l\geq 0}p^{a+l,b+l}_\lambda$$
for any $a+b\geq n+s$. Moreover, since
$h^{a,b}_\lambda=h^{n+s-a,n+s-b}_{\bar{\lambda}}$, the system of Hodge numbers
$\{h^{a,b}_\lambda\}_{a,b}$ is equivalent with the system of primitive Hodge numbers
$\{p^{a,b}_\lambda\}_{a+b\geq n+s}$.

The point is that by \cite[Theorem~6.1]{Nem-real} one has the following isomorphism of variation structures:
\[\mv_f \simeq\bigoplus_\lambda \bigoplus_{2n\geq a+b\geq n+s}\ p_{\lambda}^{a,b} \cdot\mw^{a+b+1-n-s}_\lambda ((-1)^b).\]
In particular, for any $k\geq 1$ and $u=\pm 1$ one has
\begin{equation}
p_\lambda^k(u)=\sum _{\substack{ a+b=k+n+s-1 \\ (-1)^b=u}}\, p_\lambda ^{a,b}.
\end{equation}
This  fact motivates to call the numbers $p^k_\lambda(u)$ the {\it mod--2 primitive Hodge numbers }
of $f$, or, of the corresponding variation structure.

This relation with Hodge theory can be continued. Recall that for any $f$ as above one extract from the equivariant
Hodge numbers the spectrum. Now,
if $\mv$ is a variation structure associated to an isolated hypersurface singularity,
then $Sp$ (defined in (\ref{d:mod2spec})) is the spectrum of the singularity
reduced modulo 2, i.e. if $\alpha$ belongs to the spectrum, then $\alpha\in (0,2]$ (mod 2) belongs to $Sp$. In the case of isolated curve
singularities $Sp$ is just the spectrum of the singular germ.

Obviously, in general,
$Sp$ does not always contain enough information to recover $\mv$. However,
if all the monodromy eigenvalues are different, then the dimension of all Jordan blocks is one,
 and $\mv$ is determined by $Sp$. This simple case contains for example all
spectra of cuspidal plane curve singularities.

\vspace{2mm}

At the end of this subsection we recall the connection of the variation structures with the Seifert form.

\begin{definition}\label{d:Seifert}
Let us be given two cycles $\alpha,\beta\in \tilde{H}_n(F)$. The \emph{Seifert form} of the Milnor fibration is defined to be
\begin{equation}\label{eq:LSeif}
S(\alpha,\beta)=L(\alpha,\Gamma_{1/2}\beta),
\end{equation}
where $L$ is the linking number of two $n$-dimensional cycles in $S^{2n+1}$.
\end{definition}
There is a standard fact that $S(\alpha,\beta)=\langle V^{-1}(\alpha),\beta\rangle$, where $\langle\cdot,\cdot\rangle$
is the Lefschetz pairing $\tilde{H}_{n}(F,\partial F)\times \tilde{H}_{n}(F)\to\mathbb{R}$ .
In particular, in matrix notations, the  variation operator is the transposed inverse of the
Seifert form.

\section{H--numbers for links}
\subsection{Definitions and first properties}
Let us consider $S^3$ with its standard orientation, and
let $L\in S^3$ be an {\it oriented  link}. Let $S$ be its (integral) Seifert form.
 By our convention, the Seifert form is  $S(\alpha,\beta)=L(\alpha,\beta^+)$,
 where $\alpha,\beta$ are cycles on the Seifert surface
and $\beta^+$ is the push-forward of $\beta$ in the positive direction.
This is the convention adopted by e.g. \cite{BZ,Liv,Mur-book}. Some
authors like \cite{Kau,Kaw-book} define $S(\alpha,\beta)$ as $L(\alpha^+,\beta)$.
This amounts to transposition of $S$.

Recall, that two matrices $S$ and $S'$ are \emph{congruent} if there exists an invertible matrix $A$ such that $S'=ASA^T$.

The next results were proved in \cite{keef} as  a  generalization of Theorem~12.2.9 of \cite{Kaw-book} (we would like
to thank H.~Trotter for drawing our attention to Keef's paper):

\begin{proposition}\label{prop:split}
(a) \cite[Proposition~3.1]{keef} Let $S:V\times V\to\mathbb{R}$ be a Seifert form of a link.
Then, either $S$ is real S--equivalent to the empty matrix or is real S-equivalent to $\bp S_0&0\\ 0& S_\nd\ep$, where
$\det S_\nd\neq 0$ and $S_0$ is a zero matrix.

(b) \cite[Theorem~3.5]{keef} Let us be given two matrices $S$ and $T$, which are $S$ equivalent.
Assume that they are of the form $S=\bp S_0&0\\ 0&S_\nd\ep$, $T=\bp T_0&0\\ 0&T_\nd\ep$.
Then $S_{\nd}$ and $T_\nd$ are congruent and $\dim S_0=\dim T_0$
\end{proposition}

Let us define
\[V:=(S_\nd^T)^{-1}.\]
and take the associated HVS with $\varepsilon=-1$. Its parts are the following:
 $U=\mathbb{C}^m$, where $m:=rank(V)$, $b=S_\nd-{S_\nd^T}$, and
 $h=(S_\nd^{T})^{-1}\cdot S_\nd$.

Observe that taking a conjugate  of the Seifert matrix results in an isomorphism of HVS. Hence, 
the whole structure is independent (up to an isomorphism) of the specific choice of the Seifert matrix. Hence it
is a link invariant.

\begin{definition}\label{d:mvL}
The variation structure $(U;b,h,V)$ defined above is called the \emph{variation structure} of the link $L$ and
is denoted by $\mv_L$.
\end{definition}

According to Definition~\ref{d:pklambda}, we can define the numbers
$\{q_\lambda^{k}\}_{|\lambda|<1}$ and $\{p^k_\lambda(\pm 1)\}_{\lambda\in S^1}$
of the corresponding HVS.

\begin{definition}\label{d:phn}
The numbers $\{q_\lambda^{k}\}_{|\lambda|<1}$ and $\{p^k_\lambda(\pm 1)\}_{\lambda\in S^1}$
 will be  called the \emph{H--numbers}  of the link $L$.
\end{definition}

 From basic properties of Seifert matrices we get

\begin{lemma}\label{l:symSei}
\begin{itemize}
\item[(a)] The H--numbers are symmetric in the sense that for $0<|\lambda|<1$ one has
$q^k_\lambda=q^k_{\bar{\lambda}}$, and
$$p^k_\lambda(\pm 1)=p^k_{\bar{\lambda}}(\pm(-1)^{k+s}) \ \ \ \ \ \ \mbox{for $\lambda \in S^1$}.$$
\item[(b)]H--numbers are additive with respect to the connected sum of links.
\item[(c)] If $L^{or}$ is the link $L$ with all its components with opposite orientation, then the 
H--numbers of $L$ and $L^{or}$ are the same.
\item[(d)] If $L^{mir}$ is the mirror of $L$, then the 
H--numbers are changed  as follows:
$q^{k}_\lambda(L^{mir})=q^{k}_\lambda(L)$ for any $|\lambda|<1$ and
$$p^{k}_\lambda(\pm 1)(L^{mir})=p^{k}_\lambda(\mp 1 )(L)  \ \ \ \ \ \ \mbox{for $\lambda \in S^1$}.$$
\end{itemize}
\end{lemma}

\begin{proof}
(a) follows immediately from Lemma~\ref{l:symmetry}. As for (b), observe that the Seifert matrix of the connected sum
is the direct sum of the Seifert matrices of the summands. (c) and (d)  follows from the classical facts that the
$S(L^{or})=S(L)^T$ and $S(L^{mir})=-S(L)^T$, cf. \cite{Mur-book} Propositions (5.4.6) and (5.4.7), combined with
(\ref{l:symmetry}), (\ref{l:minus}) and (\ref{l:transpose}).
\end{proof}

If $L$ is an algebraic link, i.e. a link of a plane curve singularity, it has two HVS's: the
variation structure of the singularity $\mv_f$ (see Definition~\ref{d:mvf}) and the variation  structure of
the oriented link $\mv_L$. Obviously, they agree
 $\mv_f\simeq\mv_L$, thanks to the relation $V=(S^T)^{-1}$, cf. the discussion after Definition~\ref{d:Seifert}.

\vspace{2mm}

One has very strong restrictions for
H--numbers of algebraic links: from the classical monodromy theorem
(see e.g. \cite{AVG} or \cite[Chapter 7, \S 4]{Zo})
one reads:

\begin{corollary}\label{c:montheo}
If $L$ is an algebraic link then $q^k_\lambda=0$ for any $|\lambda|<1$. Moreover, $p^k_\lambda(\pm 1)=0$ if at least one of the following conditions is satisfied
\begin{itemize}
\item $\lambda$ is not a root of unity;
\item $\lambda\neq 1$ and $k>2$;
\item $\lambda=1$ and $k>1$.
\end{itemize}
\end{corollary}

Corollary~\ref{c:montheo} admits further improvements, see e.g.
\cite[Proposition~6.14]{Nem-real}.

\begin{lemma}
If $L$ is algebraic link and $\lambda\in S^1\setminus\{1\}$, then $p^2_\lambda(-1)=0$ ($p^2_\lambda(+1)$ can be positive)
and $p^1_1(-1)=0$.
\end{lemma}

\section{Classical link invariants and H--numbers}\label{s:linkinv}
Having defined the  H--numbers, we wish
 to study their relationship with classical invariants
of the link $L$.
Recall that we have the decomposition $S=S_\nd\oplus S_0$ and the newly defined  numbers
are associated with $S_\nd$, see  (\ref{d:phn}).
\subsection{Alexander polynomial}\label{ss:ALEX}
 Define the polynomial $P(t)\in{\mathbb R}[t]$ by
$$P(t):= \prod_{|\lambda|=1} (t-\lambda)^{\sum_{k,u}kp^k_\lambda(u)}
\prod_{0<|\lambda|<1}\left((t-\lambda)(t-1/\bar{\lambda})\right)^{\sum_{k} kq^k_\lambda}.$$

\begin{lemma}\label{l:alex0}
The Alexander polynomial $\Delta(t)$ is zero if
$S_0\not=0$, and it equals $P(t)$ (up to an invertible element of $\mathbb{R}[t,t^{-1}]$) otherwise.
In this second case, the degree of $\Delta(t)$
is equal to the cardinality of the extended spectrum $ESp$.
\end{lemma}
\begin{proof}
We have
$\Delta(t)=\det(S_\nd-tS_\nd^T)=\det S_\nd^T\cdot \det(h-tI).$
\end{proof}
As the Alexander polynomial of a knot has no root at $t=1$, we get:
\begin{corollary}
If $L$ is a knot then $p_1^k(\pm 1)=0$.
\end{corollary}
The symmetry property of H--numbers (Lemma~\ref{l:symmetry})
explains (once again)  the well-known property of the Alexander polynomial, namely, if we write
$\Delta=a_0+a_1t+\dots+a_{m}t^{m}$, then $a_{n}=(-1)^ma_{m-n}$.

\subsection{Higher Alexander polynomials}\label{sub:higherAlex}
Let us recall briefly the construction of the higher order Alexander polynomials (see \cite[Definition~8.10]{BZ})
via   higher order elementary ideals
of the matrix $S-tS^T$.  We remark that our
construction differs slightly from the standard one, because we consider ideals in ${\mathbb R}[t,t^{-1}]$ instead of
$\mathbb{Z}[t,t^{-1}]$ (hence we loose some information about ${\mathbb Z}$--torsion elements).

Let $\ell$ be a positive integer.
Consider an $\ell\times \ell$ matrix $H$ over $\mathbb{R}[t,t^{-1}]$.
For $0\le n<\ell$, let $E_{n}$ be
the ideal in $\mathbb{R}[t,t^{-1}]$
generated by the determinants of all $(\ell-n)\times(\ell-n)$ minors of $H$. As $\mathbb{R}[t,t^{-1}]$ is
a principal ideal domain,
the ideal $E_n$ is generated by a single element $\Delta^H_{n}(t)\in\mathbb{R}[t,t^{-1}]$. $\Delta^H_{n}(t)$ is defined
only up to an invertible element in $\mathbb{R}[t,t^{-1}]$, multiplying it by $t$ in the appropriate power we can guarantee
that $\Delta^H_n$ is in fact a polynomial and, unless it is the zero polynomial, that $\Delta^H_n(0)\neq 0$.
\begin{definition}\label{d:alek}
The polynomial $\Delta^H_{n}(t)$ for $H=S-tS^{T}$ is called the $n$-th Alexander polynomial of the link $L$ and denoted by
$\Delta_n(t)$.
\end{definition}
The indexing was chosen so that $\Delta_0$ is the standard Alexander polynomial.
If $m_0$ is the rank of $S_0$, then $\Delta_n=0$ for $0\leq n<m_0$ and  $\Delta_{m_0}=P$, cf. (\ref{l:alex0}).

Our goal now is to express $\Delta_n$ in terms of the primitive  numbers. Notice that multiplying $H$ by a non-degenerate
matrix independent of $t$ or taking its transpose does not change the polynomials $\Delta_n$. Therefore,
$\Delta_{n+m_0}(S-tS^T)=\Delta_n(h-tI)$.
By choosing a suitable basis of $U$ we may also assume that $h$ is in the Jordan form. Moreover,
if $H=\lambda J_k-tI$,  then $\Delta^H_0=(t-\lambda)^k$ and $\Delta^H_1=1$.
Next,  we need to see what happens when we take a direct sum of several matrices.

\begin{lemma}\label{l:sum}
Let $H^1$ and $H^2$ be two square matrices, $H=H^1\oplus H^2$, and $\Delta^{H_1}_i$, $\Delta^{H_2}_j$, $\Delta^H_k$
the corresponding higher Alexander polynomials.
For fixed $\lambda\in\mathbb{C}^*$, let $a_i$ (respectively $b_j$, $c_k$) be the multiplicity of $(t-\lambda)$ in
$\Delta^{H_1}_i$ (respectively $\Delta^{H_2}_j$, $\Delta^{H_3}_k$).
Then
\[c_k=\min\{ a_i+b_j\ \colon\  i+j=k\}.\]
\end{lemma}

\begin{proof}
It is enough to use the fact that for arbitrary
minors $A_1$, $A_2$ of $H^1$ and $H^2$, $A_1\oplus A_2$ is a minor of $H$. Moreover, any minor
of $H$, with non-zero determinant, arises in this way. 
\end{proof}

\begin{remark}
Lemma~\ref{l:sum} works if some Alexander polynomials $\Delta^{H_1}_i$ or $\Delta^{H_2}_j$ are identically zero. We only have to agree that
the multiplicity of $(t-\lambda)$ in a zero polynomial is $+\infty$. 
\end{remark}

Fix  $\mu\in\mathbb{C}^*$ and set for each $k\geq 1$
\begin{equation}\label{eq:sk}
s_k(\mu):=\left\{\begin{array}{ll}
p^k_\mu(+1)+p^k_\mu(-1) & \mbox{if $\mu\in S^1$},\\
q^k_\lambda & \mbox { if $\mu\in \{\lambda,1/\bar{\lambda}\}, \ \ (|\lambda|<1$)}.
\end{array}\right.
\end{equation}
I.e., $s_k(\mu)$ is the number of Jordan blocks of size $k$ with eigenvalue $\mu$.
Now let $\Theta:=\{\theta_1,\ldots,\theta_r\}$, $\theta_1\leq \cdots \leq \theta_r$,
 be a set of integers, such that each $k\in\mathbb{Z}$ is contained in $\Theta$
precisely $s_k(\mu)$ times (hence  $r=r(\mu)=\#\Theta=\sum s_k(\mu)$).
Define the function
\[
I(n)=
\begin{cases}
\sum_{i=1}^{r(\mu)-n}\theta_i&\text{ for $n< r(\mu)$}\\
0&\text{ otherwise}.
\end{cases}
\]
The above facts combined provide:
\begin{proposition}\label{p:highalex}
The multiplicity of the root $\mu$ in the $n$--th Alexander polynomial $\Delta^h_n$ is  $I(n)$.
\end{proposition}

In Lemma~\ref{l:alex0} and Proposition~\ref{p:highalex} the exponents of the monomials $(t-\lambda)$ depended
on the sums $p^k_\lambda(+1)+p^k_\lambda(-1)$. This, together with Lemma~\ref{l:symSei}, explains
in this terminology,
why the higher Alexander polynomials of a link and its mirror are the same.

\subsection{Rational Nakanishi index}
We begin with recalling the definition of the Nakanishi index (see e.g. \cite[Section~5.4]{Kaw-book}).
Let $\Lambda=\mathbb{Z}[t,t^{-1}]$ be a ring of Laurent polynomials with integer coefficient and
$\Lambda_{\mathbb{Q}}=\mathbb{Q}[t,t^{-1}]$. For a knot $K$, set $X=S^3\setminus K$ and let us consider
the Alexander module of $K$, i.e. the homology group of $X$
\[H=H_1(X;\Lambda)\]
with coefficients in $\Lambda$. This group can be regarded as the homology group of the universal abelian
cover of $X$. It has a natural structure of a $\Lambda$ module, where $t$ and $t^{-1}$ are deck
transformations.
\begin{definition}
A \emph{square presentation matrix} for $H$ is a square matrix $A$ with entries in $\Lambda$
such that $H=\Lambda^n/A\Lambda^n$, where $n$ is the size
of $A$. The \emph{Nakanishi index} $n(K)$ is the minimal size of a square presentation matrix of the module $H$.
\end{definition}
Since we are allowed to perform row operations on a square presentation matrix and,
independently, column operations, we can
always assume that $A$ is diagonal.

It is well known \cite[Proposition~5.4.1]{Kaw-book},
that if $S$ is a Seifert matrix of $K$, then $tS-S^T$ is a square presentation matrix for $H$.
However,  in
general, its size is not minimal possible. For example, for all torus knots $n(K)=1$.

 We show a relationship between
our primitive  numbers and the Nakanishi index defined over rational numbers instead of integers.
\begin{definition}
The \emph{rational Nakanishi index} $n_{\mathbb{Q}}(K)$ is a minimal size of a
square matrix $A_{\mathbb{Q}}$ with entries in $\Lambda_{\mathbb{Q}}$ such
that
\[H\otimes\mathbb{Q}=\Lambda^n_{\mathbb{Q}}/A_{\mathbb{Q}}\Lambda_{\mathbb{Q}}^n.\]
\end{definition}
Obviously we have $n(K)\ge n_{\mathbb{Q}}(K)$.
In \cite{Nak} is proved that $n(K)$ is a lower bound for the unknotting number, hence
$n_{\mathbb{Q}}(K)$ is a lower bound for it, too. Moreover, $n_{\mathbb{Q}}(K)$ is related to the Alexander polynomials in a
following way.

\begin{proposition}\label{prop:nak1}
If $\Delta_0(K),\dots,\Delta_n(K)$ are higher order Alexander
polynomials with $\Delta_0$ the ordinary Alexander polynomial, then
\[n_{\mathbb{Q}}=\min\{k\colon\Delta_k(K)\equiv 1\}.\]
In particular $n_{\mathbb{Q}}$ is the maximal number of
Jordan blocks of the monodromy matrix with the same eigenvalue
\begin{equation}\label{eq:nqdef}
n_{\mathbb{Q}}=\max_\lambda r(\lambda)=\max\left(\max_{|\lambda|=1}
\sum_{k,u}p^k_{\lambda}(u),\max_{0<|\lambda|<1}\sum_k q^k_\lambda\right).
\end{equation}
\end{proposition}
\begin{proof}
First of all observe that given a square presentation matrix $A_{\mathbb{Q}}$ of size $n$, $\Delta_l(K)$
is the generator of ideal spanned by all $(n-l)\times(n-l)$ minors of $l$ (see Section~\ref{sub:higherAlex}).
Hence, if the $l-$th Alexander polynomial $\Delta_l$ is non-trivial, it follows that the size of $A_{\mathbb{Q}}$ is
at least $l$.

Conversely, if $\Delta_{k-1}\not\equiv 1$ and $\Delta_k\equiv 1$ we may define $A_{\mathbb{Q}}$ to be a diagonal $k\times k$ matrix with
$\Delta_{l-1}/\Delta_{l}$
on the $(l,l)$-th place. Then $\Lambda_{\mathbb{Q}}^k/A_{\mathbb{Q}}\Lambda_{\mathbb{Q}}^k$ is
easily
seen to be
isomorphic as a $\Lambda_{\mathbb{Q}}$-module
to $\Lambda_{\mathbb{Q}}^n/(tS-S^T)$.

Equation~\eqref{eq:nqdef} follows now from  Proposition~\ref{p:highalex}.
\end{proof}

\subsection{Signatures}
Besides Alexander polynomials, the Tristram--Levine signatures can also be computed
from the H--numbers.
We begin by recalling their definition.
\begin{definition}
Let $L$ be a link and $S$ a Seifert matrix of $L$.
The \emph{Tristram--Levine signature} (or the \emph{signature function}) is the function associating to each $\zeta\in S^1\setminus\{1\}$ the
signature $\sigma(\zeta)$ of the Hermitian form given by
\begin{equation}\label{eq:sigform}
M_S(\zeta):=(1-\zeta)S+(1-\bar{\zeta})S^T.
\end{equation}
The \emph{nullity} $n(\zeta)$ is the nullity of the above form (i.e. $\dim\ker M_S(\zeta)$).
\end{definition}
\begin{remark}
Some authors,
like \cite[Definition~3.11]{Mur}, define $n(\zeta)$ as the nullity increased by $1$. It is merely a matter of convention, we stick to
the notation we find more common.
\end{remark}
Clearly, in the definition of $\sigma$ (but not $n(\zeta)$) one can replace $S$ by $S_\nd$. Hence, in the sequel, for the simplicity of the notations, $S$
will denote $S_\nd$. Then, $M_S(\zeta)$ equals
\begin{equation}\label{eq:ms}
S((\zeta\bar{\zeta}-\zeta)I+(1-\bar{\zeta})S^{-1}S^T))=
(1-\bar{\zeta})S\cdot (h^{-1}-\zeta I).\end{equation}

It is not hard to express these signatures by H--numbers:
 we compute the signature function associated with each  irreducible simple HVS
and then we use the additivity of signatures.
Notice that for non--real matrices $S$, $M_S(\zeta)$ in (\ref{eq:sigform}) should be replaced by
$(1-\zeta)S+(1-\bar{\zeta})\bar{S}^T$.

\begin{lemma}\label{l:m2ksig}
Let $V$ be the variation operator  of
 $\mv_\lambda^{2k}$ (see Example~\ref{ex:mv2k}). Let $S=(\bar{V}^T)^{-1}$.
Then the signature of  $M_S(\zeta)$
is zero and the form is non-degenerate for any $\zeta$.
\end{lemma}

\begin{proof}  The non--degeneracy  follows from (\ref{eq:ms}).
For vanishing of the signature notice
 that $M_S(\zeta)$ has the block form $\left(\begin{matrix}0&A\\\bar{A}^T&0\end{matrix}\right)$ with
 $A$  non--degenerate. Hence the signature is zero by
elementary linear algebra.
\end{proof}

The case of HVS $\mw^k_\lambda(u)$ for $|\lambda|=1$
is slightly more complicated.

\begin{lemma}\label{l:null}
Set $S=(\bar{V}^T)^{-1}$, where $V$  is the variation operator  of  $\mw_\lambda^k(u)$. Then the form $M_S(\zeta)$
is non-degenerate for all $\zeta\neq \bar{\lambda}$. If $\zeta=\bar{\lambda}$ then it has a
one--dimensional kernel. In particular,
 the nullity   of the link $L$ is equal to 
\[n(\zeta)=\sum_{k,u}p^k_{\bar{\zeta}}(u)+\dim S_0.\]
\end{lemma}

\begin{proof} The first part follows from \eqref{eq:ms}. To show the formula for $n(\zeta)$ it is enough to observe that if we decompose $S=S_0\oplus S_\nd$,
and write $M_S$ and $M_{S_\nd}$ for corresponding matrices \eqref{eq:sigform}, then $\dim\ker M_S-\dim\ker M_{S_\nd}=\dim S_0$.
\end{proof}

\noindent The next result is computational. To formulate it  we need the next

\begin{convention}
Let $\alpha,\beta\in S^1$. We say that
\[\alpha<\beta\]
if $\alpha=e^{2\pi i x}$, $\beta=e^{2\pi i y}$ with $x,y\in[0,1)$ and $x<y$.
\end{convention}

\begin{proposition}\label{p:signatures}
Let $L$ be a link, and consider the primitive numbers of
 the variation structure $\mv_L$ associated with its Seifert for as above.

Let $\zeta\in S^1\setminus\{1\}$. Then the Tristram--Levine signature of $L$ is equal to
\[
\sigma(\zeta)=\sigma(\bar{\zeta})=
-\sum_{\substack{\lambda<\zeta\\k\text{ odd}\\u=\pm 1}}up^k_\lambda(u)+
\sum_{\substack{\lambda>\zeta\\k\text{ odd}\\u=\pm 1}}up^k_\lambda(u)+\sum_{\substack{k\text{ even}\\u=\pm 1}}up^k_{\zeta}(u).
\]
\end{proposition}
\begin{proof}[Sketch of proof]
By additivity of signatures under the direct sum, it is enough to prove the statement if $\mv_L=\mv_\lambda^k(u)$ for some $\lambda\in S^1$, $k\ge 1$
and $u=\pm 1$.
Let $V$ be the variation operator corresponding to $\mv_\lambda^k(u)$. By Lemma~\ref{l:BB}, $V$ is right diagonal (because $b^{-1}$ is right diagonal)
so the corresponding form $M_S(\zeta)$ (see \eqref{eq:sigform}) is left diagonal.

Assume that $\lambda\neq \zeta$. Then $M_S(\zeta)$ is non-degenerate. 
If $k$ is even, we deduce that $M_S(\zeta)$ has a $k/2$-dimensional metabolic subspace, so signature of $M_S$
is zero. If $k$ is odd, the metabolic subspace is $(k-1)/2$ dimensional, so the signature of $M_S$ is $\pm 1$, more precisely, it
is equal to $\sign\det M_S(\zeta)$ which can be explicitely computed.

If $\lambda=\zeta$, then $M_S(\zeta)$ is degenerate, we can easily compute that its kernel is one dimensional. 
If $k$ is odd, it follows that the signature is equal to zero. If
$k$ is even, the signature is $\pm 1$. The precise computation of the signature in this case requires much more effort 
(one can for instance compute $V$ explicitely from the definition) and will not be shown here.
\end{proof}

\vspace{2mm}

As a corollary, if $\zeta\in S^1$ is not an eigenvalue of monodromy, the signature $\sigma(\zeta)$ can be expressed in
terms of the (mod 2) spectrum alone.

\begin{corollary}\label{c:sig}
\begin{itemize}
\item[(a)] Let $Sp$ be the (mod 2)--spectrum of a variation structure $\mv_L$ (see Definition~\ref{d:mod2spec}).
Let  $\zeta=e^{2\pi i x}$, where $x\in (0,1)$.  Then
\[\sigma(\zeta)=-\#Sp\cap(x,x+1)+\#Sp\setminus[x,x+1]+
\sum_{\substack{k\text{ even}\\u=\pm 1}}up^k_{\zeta}(u).\]
In particular, if $\zeta $  is not
an eigenvalue of the monodromy $h$ then
\[\sigma(\zeta)=-\#Sp\cap(x,x+1)+\#Sp\setminus(x,x+1).\]
\item[(b)]
Let $ESp$ be the extended spectrum of the variation structure $\mv_L$ and let $\zeta=e^{2\pi ix}$, $x\in(0,1)$.
Let $H_x=(x,x+1)\times i\mathbb{R}$ and assume that $ESp\cap\partial H_x=\emptyset$.
Then 
\[\sigma(\zeta)=-\#ESp\cap H_x+\#ESp\setminus H_x.\]
\end{itemize}
\end{corollary}

\begin{remark}
(a) Corollary \ref{c:sig}(a)  can be compared  with \cite[Proposition 1]{Li}, where the signatures of the iterated
torus knots is  be computed. In fact, the spectrum of the torus knot $(k,l)$, or equivalently, the spectrum of the
singularity  $\{x^k+y^l=0\}$, is
\[Sp=\{\frac{i}{k}+\frac{j}{l},\,1\le i\le k-1,\,\,1\le j\le l-1\}.\]

(b) The equivariant signatures of any HVS were
 computed in \cite{Nem-real} and are expressible in
terms of H--numbers (primitive equivariant Hodge numbers). However, in general,
they are not expressible in terms of the spectrum alone.

For instance, for plane curve singularities, the  equivariant signature
$\sigma_{-1}$ cannot be determined, in general,  from the spectrum. This is the case
with  the Tristram--Levine signature $\sigma(-1)$ too. In fact, $\sigma(-1)$ is the signature
of $S+S^T$. If $S$ is the Seifert form of the plane curve singularity $f(x,y)$, then
 $\sigma(S+S^T)$  is the signature of the suspension surface singularity $f(x,y)+z^2$.
 For a pair of singularities with the same spectrum but with different
 $\sigma(-1)$ see e.g.  \cite[(6.10)]{Nem-real}.

This also shows that
 H--numbers of a link are not determined by Tristram--Levine signatures and orders of
higher Alexander polynomials alone. This can be exemplified by
a model  situation as follows. Take
some $\lambda\in S^1\setminus\{1\}$ and consider structures such that
\[p^3_\lambda(+1)=p^1_\lambda(-1)=1\text{ respectively } p^3_\lambda(-1)=p^1_\lambda(+1)=1,\]
and all other $p^k$'s for this $\lambda$ are zero. The two structures 
are different, but they provide
the same contribution to signatures, and orders of zeros of subsequent Alexander polynomials at
$t=\lambda$ are in both cases $4,1,0,\dots$.

(c) On the other hand, the higher Alexander polynomials with a set of `higher equivariant signatures'
determines the set of H--numbers. They are defines as follows, cf.
\cite[(4.4)]{Nem-real}.
Let $(U;b,h,V)$ be a variation structure, let $U_\lambda\subset U$ be the generalized $\lambda$--eigenspace
of $h$, and for each integer $k\geq 1$ consider $U^{(k)}_\lambda:=\mbox{ker}((h-\lambda I)^k|U_\lambda)$.
On $U^{(k)}_\lambda/U^{(k-1)}_\lambda$ one defines a $(\pm 1)$--hermitian form by
$B^{(k)}_\lambda(x,y)=B(x,\lambda^{1-k}(h-\lambda I)^{k-1}y)$, where $B(x,y)=b(x)(\bar{y})$.
Let $n^{(k)}_\lambda$ be the dimension of $U^{(k)}_\lambda$, while
$\sigma^{(k)}_\lambda$ the signature of $B^{(k)}_\lambda$. Then the collection of the integers
$\{n^{(k)}_\lambda\}_{k,\lambda}$ is equivalent with the collection of higher order Alexander polynomials,
while the collection of pairs of integers $\{n^{(k)}_\lambda, \, \sigma^{(k)}_\lambda\}_{k,\lambda}$
characterizes completely the variation structure (i.e the real Seifert form).
\end{remark}

\section{Some examples}
\subsection{}
Let us consider a (right-handed)  trefoil with non-degenerate Seifert matrix
\[S=\left(\begin{matrix} -1 & 0\\ -1 & -1 \end{matrix}\right).\]
The variation matrix $V=(S^T)^{-1}$
and monodromy matrix $h=V\cdot (V^T)^{-1}=V\cdot S$ are
\[V=\left(\begin{matrix} -1 & 1\\ 0 & -1 \end{matrix}\right) \ \ \ \ \
h=\left(\begin{matrix} 0 & -1\\ 1 & 1\end{matrix}\right).\]
The eigenvalues of $h$ are $\lambda_1=\frac{1}{2}-\frac{1}{2}i\sqrt{3}$ and $\lambda_2=\frac12+\frac12i\sqrt{3}$.
We need to diagonalise $h$. Let us put
\[A=-\frac{1}{i\sqrt[4]{3}}\left(\begin{matrix}  -\frac12-\frac12i\sqrt{3}&-1\\ \frac12-\frac12i\sqrt{3} & 1\end{matrix}\right).\]
Then  $AhA^{-1}$ is diagonal with diagonal entries $\lambda_1$ and $\lambda_2$, and
 \[AV\bar{A}^{T}=\frac{1}{\sqrt{3}}\left(\begin{matrix} -\frac32+i\frac12\sqrt{3} & 0\\ 0 & -\frac32-i\frac12\sqrt{3}\end{matrix}\right)=
\left(\begin{matrix} -i(\lambda_1-1)&0\\ 0 & i(\lambda_2-1)\end{matrix}\right).\]
Thus the HVS of a treefoil is
\[\mv=\mv^{1}_{\lambda_1}(-1)\oplus\mv^{1}_{\lambda_2}(+1).\]
The spectrum is $\{\frac56,\frac76\}$.

\subsection{}
Let us consider the knot $8_{20}$. We have by \cite{CL}:
\[S=\left(\begin{matrix} -1& -1& -1& -1\\ 0& 0& -1& -1\\ 0& -1& 0& -1\\  0& 0& -1& 0\end{matrix}\right),\,\,
V=(S^T)^{-1}=\left(\begin{matrix}-1& 0& 0& 0\\ 0& 1& 0& -1\\ 1& -1& 0& 0\\ 1& -1& -1& 1\end{matrix}\right).\]
And
\[h=V\cdot S=\left(\begin{matrix}1& 1& 1& 1\\ 0& 0& 0& -1\\ -1& -1& 0& 0\\ -1& 0& -1& 1\end{matrix}\right).\]
The monodromy $h$ has eigenvalues $\lambda_1=\frac{1}{2}-\frac{1}{2}i\sqrt{3}$ and $\lambda_2=\frac12+\frac12i\sqrt{3}$. It
has two Jordan blocks of size~2.  Let $A$ be such matrix that $AhA^{-1}$ is in the
Jordan form. E.g.: 
\[A=\left(\begin{matrix}
-1&2&-2&\frac32-\frac32\sqrt{3}i\\
\sqrt{3}i&\sqrt{3}i&-\frac32+\frac12\sqrt{3}i&0\\
1&-2&2&-\frac32-\frac32\sqrt{3}i\\
\sqrt{3}i&\sqrt{3}i&\frac32+\frac12\sqrt{3}i&0
\end{matrix}\right).\]

Then we have
\[W=AV\overline{A}^{T}=
\left(\begin{matrix} \frac32-\frac16i\sqrt{3}&\frac12-\frac12i\sqrt{3}&0&0\\ 1&0&0&0\\ 0&0&\frac32+\frac16i\sqrt{3}&
\frac12+\frac12i\sqrt{3}\\ 0&0&1&0
\end{matrix}\right).\]
Then $W_{12}=(-1)\cdot(\lambda_1-1)$, $W_{34}=(-1)\cdot(\lambda_2-1)$. Now the size of each Jordan block is $k=2$ and $i^k=-1$. Hence
both signs in the direct sum  decompositions are '+' and
\[\mv=\mv^{2}_{\lambda_1}(+1)\oplus\mv^{2}_{\lambda_2}(+1).\]
This is with the agreement with the fact that the Tristram--Levine signature of $W$ is zero, only it is $+1$ at $\zeta=\lambda_{1,2}$.
The knot $8_{20}$ is also reversible.

\subsection{}  Consider, for any $n\neq 0$, the following link

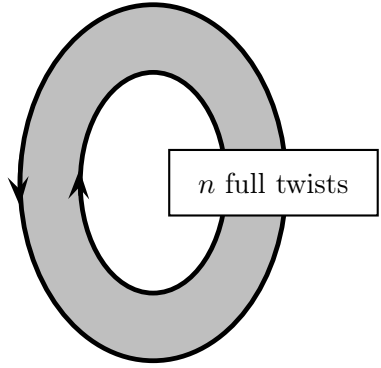
\begin{pspicture}(-6,-3)(6,3)
\psellipse[linewidth=1.8pt, fillcolor=lightgray, fillstyle=solid](0,0)(1.8,2.4)
\psellipse[linewidth=1.8pt, fillcolor=white, fillstyle=solid](0,0)(1,1.5)
\psframe[fillcolor=white, fillstyle=solid](0.2,-0.45)(3,0.45)
\rput(1.6,0){$n$ full twists}

\psellipticarc[arrowsize=8pt]{->}(0,0)(1.8,2.4){170}{190}
\psellipticarc[arrowsize=8pt]{<-}(0,0)(1,1.5){170}{190}
\end{pspicture}

This link represents two unlinks with linking number $n$.
The shaded part between the two strands forms a Seifert surface of genus $1$. The Seifert matrix in $(n)$. Hence
the variation
 structure is $\mvt^1_1(-1)$ for $n>0$ and $\mvt^1_1(+1)$ for $n<0$ and apart of that does not depend on $n$. 
Therefore we do not see the linking numbers,
or the {\it integral} Seifert form from the H--numbers.

\subsection{}
According to \cite{CL}, there are five knots with up to 12 crossings with the Alexander polynomial
\[1-4t+10t^2-16t^3+19t^4-16t^5+10t^6-4t^7+t^8=(t-\mu)^4(t-\bar{\mu})^4,\]
where $\mu=e^{\pi i/3}$.
These are $L_1=10_{99}$, $L_2=12_{n106}$, $L_3=12_{n508}$, $L_4=12_{n604}$ and $L_5=12_{n666}$.
Their monodromy matrices $(S^T)^{-1}S$ are respectively $h_1,h_2,h_3,h_4$ and $h_5$. Here we consider
 $h_1$ and $h_2$:
\begin{equation*}\label{eq:10-99}
h_1=\begin{pmatrix}
0 & 0 & -1 & 0 & 0 & 0 & 0 & 0\\
0 & 0 & 0 & 0 & 0 & 0 & -1 & 0\\
2 & 1 & 2 & -1 & 0 & -1 & 1 & 0\\
0 & 1 & 0 & 0 & 0 & 0 & 1 & -1\\
-1 & -1 & -1 & 1 & 1 & 1 & -1 & 1\\
1 & 0 & 1 & -1 & 0 & 0 & 0 & 0\\
-1 & 1 & 0 & 0 & 0 & 0 & 1 & -1\\
0 & 0 & 0 & 0 & -1 & 0 & 0 & 0
\end{pmatrix}
\end{equation*}
\begin{equation*}\label{eq:12-106}
h_2=\begin{pmatrix}0&-1&0&0&0&0&0&0\\ 1&1&-1&0&-1&-1&-1&-1\\ 1&1&1&1&0&0&0& 0\\ 0&0&-1&1&-1&-1&-1&-1\\ 0&0&0&0&0&0&0&-1\\
-1&-1&0&-1&1&1&1&1 \\ 0&0&0&0&-1&0&0&0\\ 1&1&0&1&0&-1&0&0
\end{pmatrix}
\end{equation*}
The matrix $h_{1}$ has four Jordan blocks of size two, we easily get in this case
\[p^2_\mu(+1)=p^2_\mu(-1)=p^2_{\bar{\mu}}(+1)=p^2_{\bar{\mu}}(-1)=1.\]
The matrix $h_{2}$ has a single Jordan block for each eigenvalue. We have
\[p^4_\mu(-1)=p^4_{\bar{\mu}}(-1)=1.\]

The matrices $h_3$ and $h_5$ have two Jordan blocks of size one and two of size~3. We conclude that
\[p^1_\mu(-1)=p^3_\mu(-1)=p^1_{\bar\mu}(+1)=p^3_{\bar\mu}(+1)=1.\]
In case of $h_4$ we have similarly two Jordan blocks of size one and two of size~3. We can compute that
\[p^1_\mu(+1)=p^3_\mu(+1)=p^1_{\bar\mu}(-1)=p^3_{\bar\mu}(-1)=1.\]

Hence, only the knots $12_{n508}$ and $12_{n666}$ are undistinguishable by  H--numbers.

Observe that if we take a connected sum of three left--handed treefoils and one right--handed one, then
the Alexander polynomial and the signatures for $\zeta\neq \mu,\bar{\mu}$, shall be the same as in the case
of the knots $12_{n508}$ and $12_{n666}$, but the Jordan block structure is different.

\section{Skein relations for H--numbers}
\subsection{Signatures}\label{sec:skein1}
Although we do not have a precise Skein relation for  H--numbers, there are several constrains
from them, coming mostly from relations for classical invariants.

As usually in skein relation,
we consider the three links $L_0$, $L_+$ and $L_-$.
Their Seifert matrices $S_0$, $S_+$ and $S_-$ can
be chosen (see \cite[Proof of Theorem~7.10]{Kau}) so that $S_+$ and $S_-$ are $(n+1)\times(n+1)$ matrices,
such that
\begin{equation}\label{eq:SplusSminus}
S_+-S_-=
\begin{pmatrix}
&&&\rnode{AU}{}&0\\
&\scalebox{1.5}{\underline{$\mathbf{0}$}}&&&\vdots\\
&&&&0\\
\rnode{AL}{}&&&&\rnode{AR}{}\\
0&\ldots&0&\rnode{AD}{}&1
\end{pmatrix},
\end{equation}
\ncline[nodesep=-0.5em]{AU}{AD}
\ncline[nodesep=-0.5em]{AR}{AL}
where $\underline{\mathbf{0}}$ denotes an $n\times n$ zero matrix.
Moreover $S_0$ arises from $S_+$ by deleting the $(n+1)$-st row and $(n+1)$-st column.

The following fact is classical (see \cite[Lemma~12.3.4]{Kaw-book} or \cite{Mur}).
\begin{proposition}\label{p:sigskein}
For any $\zeta\in S^1\setminus\{1\}$, we have the following bounds
\[|\sigma_{L_\pm}(\zeta)-\sigma_{L_0}(\zeta)|+|n_{L_\pm}(\zeta)-n_{L_0}(\zeta)|\le 1.\]
\end{proposition}
\begin{proof}
We follow the proof of \cite{Kaw-book}.
For fixed $\zeta$, let $M_+$, $M_-$ and $M_0$ denote the forms \eqref{eq:sigform} for $S_+$, $S_-$ and $S_0$, respectively.
Let $p_+,p_-,p_0$, $q_+,q_-,q_0$ and $n_+,n_-,n_0$ be the maximal dimension of subspaces on which $M_+,M_-,M_0$ are, respectively,
positive definite, negative definite and zero.
As $S_0$ is submatrix of $S_+$, $M_0$ is a restriction of $M_+$ onto an $n-$dimensional subspace. It follows that
\begin{align*}
p_0\le p_+\le p_++1\\
q_0\le q_+\le q_++1\\
n_0\le n_+\le n_++1.
\end{align*}
As $p_0+q_0+n_0+1=p_++q_++n_+$, the statement follows.
\end{proof}
The above proposition and (\ref{p:signatures}) give restriction for possible 
H--numbers of $L_\pm$ and $L_0$, when
the primitive  numbers of one of them are known.

\smallskip
\subsection{Semicontinuity of the extended spectrum}\label{ss:semies}
The inequality of (\ref{p:sigskein})
can be used to prove a variant  of the semicontinuity of spectra.

Here some comments are in order. The semicontinuity property of (genuine) spectrum of
hypersurface singularities says the following: if $\{f_t\,:\, t\in (\CC,0)\}$ is a family of isolated
singularities, then for any interval $I=[\alpha,\alpha+1)$ the spectral numbers $\{Spec(f_t)\}_t$ associated with
$f_t$ satisfies: $\#\,Spec(f_0)\cap I\geq \#\, Spec(f_{t\not=0})\cap I$, see \cite{St,Var}.
The semicontinuity principle is codified in  the very geometric substance of the deformation.
In particular, several  invariants behave semicontinuously, e.g.
for the Milnor number $\mu(f_0)\geq \mu(f_{t\not=0})$. If one tries to study this
phenomenon in the case of arbitrary links, one needs to assume that the geometric situation mimics
in the right way the presence of the deformation. In the next proposition we will assume that
$\deg\Delta$ (i.e. the Milnor number in algebraic case) is monotone.

\begin{proposition}\label{prop:semicont}
Let $L_1$ and $L_2$ are two links.
Let $ESp_1$ and $ESp_2$ be the corresponding extended spectra.
Let $\Delta_{L_i}$ ($i=1,2$) be the  characteristic polynomial of the monodromy operator
associated with the non--degenerate part of the Seifert form; in other words, the
first non--zero higher order Alexander polynomial. (If the Seifert forms are non--degenerate then
 $\Delta_{L_i}$ is just the Alexander polynomial. See subsections (\ref{ss:ALEX}) and (\ref{sub:higherAlex}).)

Assume also that $x\in(0,1)$ is such that
$\partial H_x\cap (ESp_1\cup ESp_2)=\emptyset$ (where $H_x=(x,x+1)\times i\mathbb{R}$ as above).
Moreover assume that one of the following holds.
\begin{itemize}
\item[(a)] $L_1$ arises from $L_2$ by changing a negative (or left-handed) crossing to positive (or right-handed)
crossing (see \cite[Exercise 3.2.5]{Liv} for the necessary definitions)
and $\deg\Delta_{L_1}\ge\deg\Delta_{L_2}$;
\item[(b)] $L_1$ arises from $L_2$ by changing one crossing and  $\deg\Delta_{L_1}>\deg\Delta_{L_2}$;
\item[(c)] $L_1$ arises from $L_2$ by a hyperbolic transformation (i.e. $L_1$ and $L_2$ can play a role
of $L_0$ and $L_\infty$ at some diagram, see \cite[Definition~12.3.3]{Kaw-book}) and $\deg\Delta_{L_1}>\deg\Delta_{L_2}$.
\end{itemize}
Then
\[\#ESp_1\cap H_x\ge \#ESp_2\cap H_x.\]
\end{proposition}
\begin{proof}
Let $a_1=\#ESp_1\cap H_x$, $a_2=\#ESp_2\cap H_x$, $b_1=\#ESp_1\setminus H_x$, $b_2=\#ESp_2\setminus H_x$.
Then by Corollary~\ref{c:sig}(b) we have
\begin{align*}
a_1+b_1&=\deg\Delta_{L_1}&a_2+b_2&=\deg\Delta_{L_2}\\
-a_1+b_1&=\sigma_{L_1}(e^{2\pi i x})&-a_2+b_2&=\sigma_{L_2}(e^{2\pi i x}).
\end{align*}
Thus
\[a_1-a_2=\frac12\left(\deg\Delta_{L_1}-\deg\Delta_{L_2}-\sigma_{L_1}(e^{2\pi ix})+\sigma_{L_2}(e^{2\pi i x})\right).\]
Now, in the case (a), as $\deg\Delta_{L_1}>\deg\Delta_{L_2}$ and $L_1$ and $L_2$ have the same number of components,
the degrees differ at least by $2$. The signatures cannot differ by more than $2$ by Proposition~\ref{p:sigskein}.
In the case (b), signature of $L_1$ is not larger than that of $L_2$, in case (c), both degree of Alexander and
signature cannot differ by more than one.
\end{proof}
The above results is enough to prove a variant of Theorem~6.7 from \cite{Bo} with $ESp\cap H_x$
instead of Tristram--Levine signatures. Since for algebraic links the $ESp$ is the same as the ordinary spectrum,
we can relate the spectra of singularities of a plane curve with the spectrum of the singularity at infinity. See
\cite{BN} for details.

\begin{remark}
In singularity theory the signature is not semicontinuous, see e.g. \cite{KN}. Hence we do not have
the semicontinuity property for each particular Hodge number either.
Similar behavior can be observed in the knot theory: the condition
$\deg\Delta_{L_1}>\deg\Delta_{L_2}$ (in the notation from Proposition~\ref{prop:semicont}) alone is not sufficient to determine
the sign of $\sigma_{L_1}(\zeta)-\sigma_{L_2}(\zeta)$, so we do not have a strong 'semicontinuity property' for signatures of knots.
We cannot also expect the semicontinuity property for each H--number of links.
\end{remark}

\subsection{Higher Alexander polynomials and rational Nakanishi index}
Next, we wish to connect the higher order Alexander polynomials
associated with $S_\pm$ and $S_0$, cf. (\ref{sec:skein1}).
In order to formulate the result,
we need to introduce some additional notation.

\begin{convention} \
\begin{itemize}
\item[(a)] Let us fix $\lambda\in\mathbb{C}$ for this section. For any matrix $H$ with coefficients
in $\mathbb{C}[t]$ we define
\[d(H)=\ord_{t=\lambda}\det H.\]
\item[(b)] For any $m\times m$ matrix $K$ and $1\le i,j\le m$, we define $K^{i,j}$ as the
 $(m-1)\times (m-1)$ minor
resulting from $K$ by removal of $i-$th row and $j-$th column.
\end{itemize}
\end{convention}
The next easy  lemma  will be important in the sequel.
\begin{lemma}\label{l:laplacetr}
For any matrix $H$ of size $m\times m$ and for any $1\le j\le m$ one has
\[d(H)\ge \min\{d(H^{i,j})\colon 1\le i\le m\}.\]
\end{lemma}
\begin{proof}$\det H=\sum_i(-1)^{i+j}H_{i,j}\det H^{i,j}$ by
the Laplace expansion of  $\det H$ along the $j-$th column.
(Recall that $H_{i,j}$ denotes the element of $H$ at $i-$th row and $j-$th column.) It follows that,
if all $\det H^{i,j}$ are divisible by $(t-\lambda)^d$, so will be $\det H$.
\end{proof}
The definition of $d(H)$ is motivated by the higher order Alexander polynomials.
Namely, if $S$ is an $m\times m$ Seifert matrix and $H=S-tS^T$, then the multiplicity of a root of the $k-$th Alexander
polynomial (see Definition~\ref{d:alek}) of $H$ at $t=\lambda$ can be expressed as the minimum of $d(K)$,
where $K$ runs through all  $(m-k)\times (m-k)$ minors of $H$:
\begin{equation}\label{eq:d(K)}
\ord_{t=\lambda}\Delta_k^H=\min\{d(K)\colon\text{$K$ is an $(m-k)\times (m-k)$ minor of $H$}\}.
\end{equation}

For Seifert matrices $S_+$, $S_-$ and $S_0$ of links $L_+$, $L_-$ and $L_0$ (notation from Section~\ref{sec:skein1})
let us define $H_*=S_*-tS^T_*$ ('$*$' is one of '$+$', '$-$' or '$0$') and
\[d_k^{*}=\ord_{t=\lambda}\Delta_k^{H_*}.\]
An immediate consequence of Lemma~\ref{l:laplacetr} is that for any $k$ and $*\in\{+,-,0\}$:
\begin{equation}\label{eq:dless}
d_{k}^{*}\ge d_{k+1}^{*}
\end{equation}

The skein relation gives the following restrictions for values of $d_k^*$.
\begin{proposition}\label{prop:skeindk}
The integers $d_k^+$, $d_k^-$ and $d_k^0$ satisfy the following relations:
\begin{subequations}
\begin{align}
d_{k}^0&\ge d_{k+1}^{\pm}\label{eq:simple}\\
d_k^\pm&\ge d_{k+1}^0\label{eq:secondone}\\
d_k^{\pm}&\ge d_{k+1}^{\mp}&\text{if $\lambda\neq 1$}\label{eq:thirdone}\\
d_k^{\pm}&\ge \min(d_{k+1}^{\mp}+1,d_{k}^{\mp})&\text{if $\lambda=1$}.\label{eq:fourthone}
\end{align}
\end{subequations}
\end{proposition}
\begin{proof}
Let  $m$ be the size of $H_0$.
$H_0$ can be regarded as an $m\times m$ minor of both $H_+$ and $H_-$ (cf. \eqref{eq:SplusSminus}).
Then any $(m-k)\times (m-k)$
minor of $H_0$ is also an $((m+1)-(k+1))\times ((m+1)-(k+1))$ minor of $H_+$ and of $H_-$.
It follows that $d_{k+1}^\pm\le d_k^0$, because in $d_{k+1}^\pm$ we take a minimum over larger set.
Equation~\eqref{eq:simple} follows.

As for \eqref{eq:secondone}, we can divide all possible $(m+1-k)\times (m+1-k)$ minors of $H_+$
in three categories.
\begin{itemize}
\item $A^k_\alpha$, $\alpha\in\mathcal{A}$, will denote minors lying entirely in $H_0$;
\item $B^k_\beta$, $\beta\in\mathcal{B}$, will denote minors containing a part of the last column or row of $H_+$
but not containing the corner;
\item $C^k_\gamma$, $\gamma\in\mathcal{C}$,
will denote minors containing the element of $H_+$ lying in $(m+1)$st row and $(m+1)$st column.
\end{itemize}

Graphically we can present these minors like that

\begin{pspicture}(-5,-3)(5,3)
\rput(0,0){
\psscalebox{2}{$%
\left[%
\begin{matrix}%
&&&&&&\\&&&&&&&\\&&&&&&\\&&&&&&\\&&&&&&\\&&&&&&%
\end{matrix}%
\right]%
$%
}}
\rput(-2,-2.7){$H_+$}
\rput(-0.25,0.25){
\psscalebox{1.65}{$%
\left[%
\begin{matrix}%
&&&&&&\\&&&&&&&\\&&&&&&\\&&&&&&\\&&&&&&\\&&&&&&%
\end{matrix}%
\right]%
$%
}}
\rput(-1.8,-1.6){$H_0$}
\psframe[fillcolor=lightgray,opacity=0.5,fillstyle=solid,linewidth=1pt, framearc=0.3](-2,0.5)(-0.5,2)
\rput(-1,0.2){$A$}
\psframe[fillcolor=lightgray,opacity=0.5,fillstyle=solid,linewidth=1pt, framearc=0.3](1.1,0)(2.6,1.5)
\rput(0.7,0){$B$}
\psframe[fillcolor=lightgray,opacity=0.5,fillstyle=solid,linewidth=1pt, framearc=0.3](1,-2.7)(2.5,-1.2)
\rput(1,-1){$C$}
\end{pspicture}

\smallskip

\noindent By \eqref{eq:d(K)} one has (where $d(\alpha)=d(A^k_\alpha)$, and similarly
for $d(\beta)$ and $d(\gamma)$):
\begin{equation}\label{eq:DDD}
d_{k}^+=\min(\min_{\alpha\in\mathcal{A}} d(\alpha),
\min_{\beta\in\mathcal{B}}d(\beta),\min_{\gamma\in\mathcal{C}}d(\gamma)).
\end{equation}
We need to show that all three minima are greater than $d_{k+1}^0$.

First of all $\min_{\alpha\in\mathcal{A}}d(\alpha)$ is
precisely $d_{k-1}^0$. 

Let $B_{\beta}^{j,m+1-k}$ be
the minor of $H_+$ resulting by removing the last column and $j-$th row from $B_\beta$ (assuming that $B_\beta$ contains
a part of $(m+1)$st column of $H_+$ as on the picture, not a part of $(m+1)$st row). Remark that $B_{\beta}^{j,m+1-k}$ is
an $(m-k)\times (m-k)$ minor of $H_0$.
By Lemma~\ref{l:laplacetr} we have
\[d(B_\beta)\ge\min_{j=1,\dots,m+1-k} d(B_{\beta}^{j,m+1-k}).\]
Thus
\[\min_{\beta\in\mathcal{B}}d(B_\beta)\ge d_k^0.\]
In order to deal with $d(C_\gamma)$ we need to do expand $\det C_\gamma$ first along the last row and
then along the last column.
Applying Lemma~\ref{l:laplacetr} twice we get
\[d(C_\gamma)\ge\min_{1\le j,j'\le m+1-k} d(C_{\gamma}^{(j,m+1-k),(j',m+1-k)}),\]
where $C_{\gamma}^{(j,m+1-k),(j',m+1-k)}$ arises from $C_\gamma$ by deleting $j-$th and $(m+1-k)-$th column and $j-$th and
$(m+1-k)-$th row. It is thus an $(m-(k+1))\times (m-(k+1))$ minor of $H_+$ and also of $H_0$, because it does not contain
neither the last column, nor the last row of $H_+$.
Hence,
$d(C_\gamma)\ge d_{k+1}^0$.
Finally,
using \eqref{eq:d(K)} and (\ref{eq:DDD}) we obtain
$d_k^+\ge
\min(d^0_{k-1},d^0_k,d^0_{k+1})=d^0_{k+1}$.
Hence the inequality \eqref{eq:secondone} is proved.

In order to prove the last two inequalities, let us consider two $(m+1-k)\times(m+1-k)$ minors
$K_+$ and $K_-$ of $H_+$ and $H_-$
obtained by removing the same columns and the same rows from matrices $H_+$ and $H_-$.

As $H_+-H_-$ is a matrix with $(t-1)$ in the place $(m+1,m+1)$, and zeros everywhere else,
$K_+=K_-$ unless they contain the element at the
bottom right corner of $H_+$ and $H_-$. If they do not contain,
\[\det K_+=\det K_-,\text{ so } d(K_+)=d(K_-)\]
If they do,
\[\det K_+=\det K_-+(t-1)\det K_0,\]
where $K_0=K_\pm^{m+1-k,m+1-k}$ arises from $K_\pm$
by removing the last column and the last row. In this case we deduce that
\[d(K_+)\ge \min(d(K_-),d(K_0)+s),\]
where $s=1$ if $\lambda=1$ and $0$ otherwise.

This shows in particular that
\[d_k^+\ge \min (d_k^-,s+d_k^0).\]
Now it is enough to observe that by \eqref{eq:secondone} $d_k^0\ge d_{k-1}^-$.
\end{proof}

In order to apply this skein relation, let us fix $\lambda$ with $0<|\lambda|\le 1$
and
consider the set $\Theta$ (defined after the proof of Lemma~\ref{l:sum}) associated with $\lambda$.
For any $N\geq 1$ set
\begin{equation}\label{eq:PN}
P_N=\#\{\theta\in\Theta:\theta\ge N\}.\end{equation}
$P_N$ can be interpreted as the number of the Jordan blocks of
size at least $N$ with eigenvalue $\lambda$ of the monodromy
matrix $h$; i.e., with the notation of \eqref{eq:sk}, one has:
\[P_N=\sum_{k\ge N} s_k(\lambda).\]

\begin{corollary}\label{cor:PN} Fix some $\lambda$.
Let $P_N^+$ and $P_N^-$ be the $P_N$ numbers associated to the links $L_+$ and $L_-$ respectively
as in (\ref{eq:PN}).
Then for any $N\geq 2$ one has
\[|P_N^+-P_N^-|\le 2N\]
while for $N=1$, $|P_1^+-P_1^-|\le 1$.
\end{corollary}
\begin{proof}
For $N>1$, assume that $P_N^+-P_N^-=a>0$. By Proposition~\ref{p:highalex}  for any $i<P_N^+$ we get
\begin{equation}\label{eq:PNdi}
d_i^+-d_{i+1}^+\ge N,
\end{equation}
and  for $i\ge P_N^-$
\begin{equation}\label{eq:PNdi2}
d_i^--d_{i+1}^-\le N-1.
\end{equation}
Therefore, we obtains the next sequence of inequalities:
\begin{multline*}
d^+_{P_N^+-1}\stackrel{(*)}{\ge} d^-_{P_N^+}\stackrel{(**)}{\ge} d^-_{P_N^-}-a(N-1)\stackrel{(*)}{\ge} \\
\ge d^+_{P_N^-+1}-a(N-1)\stackrel{(***)}{\ge}d^+_{P_N^+-1}+(a-2)N-a(N-1).
\end{multline*}
Here the inequalities denoted by ($*$) follow
from \eqref{eq:thirdone}, ($**$) from \eqref{eq:PNdi2} and ($***$) from \eqref{eq:PNdi}.
Hence $(a-2)N-a(N-1)\le 0$,  or $2N\geq  a$.

So now assume that $N=1$. Then $P_1^\pm=\min\{i>0\colon d_i^\pm=0\}$. So let us take $k$ such that
$d_k^->0=d_{k+1}^-$ (i.e. $P_1^-=k+1$).
By \eqref{eq:thirdone} we have $d_{k-1}^+\ge d_{k}^-$, hence $d_{k-1}^->0$, so $P_1^+\ge k$. The argument follows
from symmetry.
\end{proof}

For knots,
the maximum of the values $P_1$ for all $\lambda\neq 1$ is, by Proposition~\ref{prop:nak1}, equal to the Nakanishi
index $n_{\mathbb{Q}}$. Therefore, Corollary~\ref{cor:PN} implies that

\begin{corollary}
Let $K_+$ and $K_-$ be two knots differing by one change of crossing. Then
\[|n_{\mathbb{Q}}(K_+)-n_{\mathbb{Q}}(K_-)|\le 1.\]
\end{corollary}

In particular, we reprove a theorem of Nakanishi in a weaker version:
\begin{corollary}
The rational Nakanishi index of a knot $K$ is bounded from above by the unknotting number of $K$.
\end{corollary}

\end{document}